\documentclass[10pt]{amsart}
\usepackage{amssymb}
\usepackage{epsfig}
\usepackage{url}
\usepackage[T1]{fontenc}
\theoremstyle{plain}

\newtheorem{thm}{Theorem}[section]
\newtheorem{cor}[thm]{Corollary}
\newtheorem{lem}[thm]{Lemma}
\newtheorem{prop}[thm]{Proposition}
\newtheorem{rem}[thm]{Remark}

\newtheorem{conj}[thm]{Conjecture}

\newtheorem{prob}[thm]{Problem}

\def\cal{\mathcal}
\def\bbb{\mathbb}
\def\op{\operatorname}
\renewcommand{\phi}{\varphi}

\newcommand{\N}{\bbb{N}}

\newcommand{\id}{\mathrm{id}}

\newcommand{\adj}{\mathrm{adj}}

\begin{document}

\title[Identities involving Prouhet-Thue-Morse sequence]{Some identities involving Prouhet-Thue-Morse sequence and its relatives}
\author{Jakub Byszewski, Maciej Ulas}

\thanks{The research of the authors is supported by the grant of the Polish National Science Centre no. UMO-2012/07/E/ST1/00185}

\begin{abstract}
Let $s_{k}(n)$ denote the sum of digits of an integer $n$ in base $k$. Motivated by certain identities of
Nieto, and Bateman and Bradley involving sums of the form $\sum_{i=0}^{2^{n}-1}(-1)^{s_{2}(i)}(x+i)^{m}$ for $m=n$ and $m=n+1$,
we consider the sequence of polynomials
\begin{equation*}
f_{m,n}^{\mathbf u}(x)=\sum_{i=0}^{k^{n}-1}\zeta_{k}^{s_{k}(i)}(x+{\mathbf u}(i))^{m}.
\end{equation*}
defined for sequences ${\bf u}(i)$ satisfying a certain recurrence relation.


 We prove that computing these polynomials is essentially equivalent with computing their constant term and we find an explicit formula for this number. This allows us to prove several interesting identities involving sums of binary digits. We also prove some related results which are of independent interests and can be seen as further generalizations of certain sums involving Prouhet-Thue-Morse sequence.

\end{abstract}

\maketitle

\section{Introduction}\label{Section1}
In this note, we are interested in generalizing  certain identities involving the classical Prouhet-Thue-Morse sequence $t_{n}=(-1)^{s_{2}(n)}$, where we write $s_{k}(n)$ for the sum of digits of the number $n$ written in base $k$ so that $t_n$ depends on the parity of the number of ones in the binary expansion of $n$. The Prouhet-Thue-Morse sequence arises in many areas of mathematics and physics. An interesting and varied collection of applications and properties of the Prouhet-Thue-Morse sequence can be found in a survey article of Allouche and Shallit \cite{AllSh}. Among many interesting identities satisfied by the sequence $t_n$ we have the classical identity involving the polynomial $f_{m,n}(x)$ defined by
\begin{equation*}
f_{m,n}(x)=\sum_{i=0}^{2^{n}-1}t_{i}(x+i)^{m}, \quad m,n\geq 0.
\end{equation*}
The polynomials $f_{m,n}$ vanish when $m<n$. This follows immediately from the identity $\sum_{i=0}^{2^{n}-1}t_{i}i^{m}=0$ for nonnegative integers $m<n$ obtained by Prouhet in 1850s (see also \cite[p. 37]{Fogg} and references given therein). We also have a similar identity for the polynomial
\begin{equation*}
g_{m,n}(x)=\sum_{i=0}^{2^{n}-1}t_{i}(x+s_{2}(i))^{m}=0
\end{equation*}
and integers $m<n$. This is actually much simpler to prove. Furthermore, we have some curious identities for $f_{m,n}(x)$ when $m=n$ and $m=n+1$, namely
\begin{equation*}
f_{n,n}(x)=(-1)^{n}n!2^{\frac{n(n-1)}{2}}, \quad f_{n+1,n}(x)=(-1)^{n}(n+1)!2^{\frac{n(n-1)}{2}}\Big(x+\frac{2^{n}-1}{2}\Big).
\end{equation*}
The proof of the former identity was given by Nieto in \cite{Knuth}. The proof of the latter identity was obtained by Bateman and Bradley \cite{BatBra}. One can also prove an analogous result for the polynomials $g_{m,n}$, namely
\begin{equation*}
g_{n,n}(x)=(-1)^{n}n!, \quad g_{n+1,n}(x)=\frac{(-1)^{n}}{2}(n+1)!(2x+n).
\end{equation*}

The aim of this article is to provide a context in which all these equalities can be obtained as special cases of a more general result.  To be more precise, let us observe that the sequences $s_{2}(n)$ and $n$ both satisfy a recurrence relation: \begin{equation}
\mathbf{u}(2i+j)=p \mathbf{u}(i)+j\quad\mbox{for}\quad j=0,1.
\end{equation}
(with $p=1$ for the sequence $\mathbf{u}(n)=s_{2}(n)$ and $p=2$ for the sequence $\mathbf{u}(n)=n$). This suggests to consider a family of polynomials
\begin{equation*}
f_{m,n}^{\mathbf u}(x)=\sum_{i=0}^{k^{n}-1}\zeta_{k}^{s_{k}(i)}(x+{\mathbf u}(i))^{m},
\end{equation*}
where $k\geq 2$ is an integer, $\zeta_k\neq 1$ is a $k$-th root of unity and $\mathbf{u}(n)$ is a sequence of elements of $V$ which satisfies a recurrence relation
\begin{equation*}
\mathbf{u}(ki+j)=P(\mathbf{u}(i))+jq\quad\mbox{for} \quad i \geq 0, j=0,1,\ldots,k-1,
\end{equation*}
and $q\in V$, where $V$ is a finitely dimensional vector space over a field $K$ and $P\colon V\to V$ is a linear endomorphism. The polynomial $f_{m,n}^{\bf u}$ with a suitably chosen sequence ${\bf u}$ can be seen as a generalization of the polynomials $f_{m,n}$ and $g_{m,n}$ considered before.

We now briefly describe the contents of the paper. In Section 2, we study the polynomials $f_{m,n}^{\mathbf u}(x)$. We start by reducing the problem of computing the polynomial $f_{m,n}^{\bf u}(x)$ to computing its constant term $F^{\bf u}_{m,n}:=f_{m,n}^{\bf u}(0)$. Then, using a simple induction argument, we prove that $F^{\bf u}_{m,n}=0$ for $m<n$. Next, we present several results concerning the computation of $F^{\bf u}_{m,n}$ for $m\geq n$. In particular, we state a recurrence relation satisfied by $F^{\bf u}_{m,n}$  (Theorem \ref{thm1}) and this allows us to compute it in closed form (Theorem \ref{czapka} and formula (\ref{defH}), see also Corollary \ref{closedF}). As a consequence, we get several curious identities involving the function $s_{2}(n)$.

In Section 3, we consider the function $N_{k}(l,i)$ which counts the number of occurrences of a digit $l$ in the expansion of an integer $i$ in base $k$. We define a related function with $k=2b-1$ given by $N_{k,v}(i)=\sum_{j=1}^{b}N_{k}(v_{j},i)$, where $v=(v_{1},\ldots,v_{b})$ is a $b$-tuple of pairwise distinct elements of $\{1,\ldots,2b-1\}$. Using some polynomial identities and a simple differential operator, we prove the identity
$$
\sum_{i=0}^{k^{n}-1}(-1)^{N_{k,v}(i)}i^{m}=0,
$$
where $m<n$ (Theorem \ref{PTMlike}). This identity allows us to think about the sequences $N_{k}(l,i)$ (for a fixed $k$ and $l$) as generalizations of the Prouhet-Thue-Morse sequence. We also state a simple expression for the value of this sum when $m=n$ (which is non-zero in general).

Finally, in the last section we propose several conjectures generalizing the identities obtained in previous sections to the realm of polynomials in several variables.

\section{First results}\label{Section2}

We start by defining a class of sequences for which an analogue of Prouhet's result holds.

Let $k\geq 2$ be an integer. By far the most important case will be when $k=2$, but we will state most of the results for a general $k$. Let $K$ be a field and let $V$ be a finitely dimensional vector space over $K$. We consider sequences $\mathbf{u}(n)$, $n\geq 0$ which take values in $V$ and satisfy a recurrence relation
\begin{equation}\label{czajnik}
\mathbf{u}(ki+j)=P(\mathbf{u}(i))+jq\quad\mbox{for}\quad i\geq 0,  j=0,1,\ldots,k-1 \quad \text{and} \quad \mathbf{u}(0)=0,
\end{equation}
where $q\in V$ and $P\colon V\to V$ is a linear map. We will be mostly interested in the case when $V=K$, but sometimes (in particular in Corollarries \ref{cor2} and \ref{cor3}) we will use the more general setup. We have already noted in the introduction that the sequence $\mathbf{u}(n)=s_{2}(n)$ satisfies such a recurrence with $k=2$, $V=K$, $P=\id$, $q=1$ and the sequence $\mathbf{u}(n)=n$ satisfies such a recurrence with $V=K$, $q=1$ and $P=$ multiplication by $k$.

The assumption that $\mathbf{u}(0)=0$ is not particularly restrictive. In fact, for any sequence $\mathbf{u}$ satisfying the recurrence in ($\ref{czajnik}$), substituting $i=j=0$ shows that $\mathbf{u}(0)$ is a fixed point of $P$, and so after subtracting a fixed $P$-invariant vector, we can always get $\mathbf{u}(0)=0$. In what follows, we always assume that $\mathbf{u}(0)=0$. We will show that all the identities stated in the introduction generalize to the more general context of sequences satisfying recurrence (\ref{czajnik}).

\begin{rem}
{\rm If sequences $\mathbf{u}_1,\mathbf{u}_2,\ldots,\mathbf{u}_{\nu}$ satisfy recurrences of the form (\ref{czajnik}) with $V_i$, $P_i\colon V_i \to V_i$ and $q_i\in V_i$, then the sequence $$\mathbf{u}=(\mathbf{u}_1,\mathbf{u}_2,\ldots,\mathbf{u}_{\nu})$$ satisfies such a recurrence with $V=V_1\oplus\ldots\oplus V_{\nu}$, $P=P_1\oplus\ldots\oplus P_{\nu}$, and $q=(q_1,\ldots,q_{\nu})$. }
\end{rem}

We are ready to introduce the family of polynomials which will be the main object of our study. Let $k\geq 2$ and $\zeta_{k}\neq 1$ be a (not necessarily primitive) $k$-th root of unity (we assume $\zeta_k \in K$). For integers $m,n\geq 0$, consider the polynomials
\begin{equation}\label{defpol}
f_{m,n}^{\mathbf u}(x)=\sum_{i=0}^{k^{n}-1}\zeta_{k}^{s_{k}(i)}(x+{\mathbf u}(i))^{m}.
\end{equation}
The polynomial $f_{m,n}^{\mathbf u}$ is a polynomial in $x$ with values in the symmetric algebra $\mathrm{Sym}(V)$ of $V$. If we choose a basis $v_1,\ldots,v_{\nu}$ of $V$, we can regard $f_{m,n}^{\mathbf u}$ as a polynomial in $x$ whose coefficients are homogenous polynomials in $v_1,\ldots,v_{\nu}$.  The polynomials considered in the introduction are all special cases of $f_{m,n}^{\mathbf u}$ (for $\mathbf{u}(n)=n$ and $\mathbf{u}(n)=s_2(n)$). For future use, note that the endomorphism $P\colon V\to V$ induces an endomorphism of $\mathrm{Sym}(V)$, which we will denote by the same letter. In what follows, we always assume that the sequence $\mathbf{u}$ satisfies the recurrence (\ref{czajnik}).


We start with the following simple lemma which allows us to reduce the problem of computing the polynomial $f_{m,n}^{\mathbf u}$ to the problem of computing the elements
$$F_{m,n}^\mathbf{u}=f_{m,n}^{\mathbf u}(0)=\sum_{i=0}^{k^{n}-1}\zeta_{k}^{s_{k}(i)}{\mathbf u}(i)^{m}.$$ To simplify notation, we will often write $F_{m,n}$ for $F_{m,n}^\mathbf{u}$ and $f_{m,n}(x)$ for $f_{m,n}^{\mathbf u}(x)$.

\begin{lem}\label{lem1}
For $m,n\geq 0$, the following equality holds:
\begin{equation*}
f_{m+1,n}(x)=(m+1)\int_{0}^{x}f_{m,n}(t)dt+F_{m+1,n}.
\end{equation*}
\end{lem}
\begin{proof}
We have
$$
\int_{0}^{x}f_{m,n}(t)dt=\sum_{i=0}^{k^{n}-1}\zeta_{k}^{s_{k}(i)}\frac{(t+\mathbf{u}(i))^{m+1}}{m+1}\Big|_{0}^{x}=
                        \frac{f_{m+1,n}(x)-F_{m+1,n}}{m+1}.
\qedhere
$$\end{proof}

The values of $F_{m,n}$ for $\mathbf{u}(n)=s_2(n)$ can be computed directly. Denote by $ {m \brace n}$ the Stirling number of the second kind, i.e., the number of partitions of a set with $m$ elements into $n$ nonempty subsets.

\begin{prop} Let $\mathbf{u}(n)=s_2(n)$. Then $F_{m,n}=(-1)^n n! {m \brace n}$.
\end{prop}
\begin{proof}
We have $F_{m,n}=\sum_{i=0}^{2^{n}-1}(-1)^{s_{2}(i)}s_2(i)^m$. Among integers $0\leq i \leq 2^n-1$ there are exactly $\binom n s$ many with $s_2(i)=s$. Thus we get $F_{m,n}=\sum_{s=0}^{n}(-1)^{s}\binom n s s^m$. This is equal to $(-1)^n n! {m \brace n}$ by a well-known formula.
\end{proof}

Note that in this case we have $F_{m,n}=0$ for $m<n$. In fact, this is a general phenomenon.

\begin{prop}\label{lem2}
Let $m<n$ be nonnegative integers. Then $f_{m,n}(x)=0$.
\end{prop}
\begin{proof}
It is sufficient to prove that $F_{m,n}=0$. Indeed, once we know that $F_{m,n}=0$ for $m<n$, we see that
\begin{align*}
f_{m,n}(x)&=\sum_{i=0}^{k^{n}-1}\zeta_{k}^{s_{k}(i)}(x+\mathbf{u}(i))^{m}=\sum_{i=0}^{k^{n}-1}\zeta_{k}^{s_{k}(i)}\sum_{p=0}^{m}{m\choose p}\mathbf{u}(i)^{p}x^{m-p}\\
           &=\sum_{p=0}^{m}{m\choose p}x^{m-p}\sum_{i=0}^{k^{n}-1}\zeta_{k}^{s_{k}(i)}\mathbf{u}(i)^{p}=\sum_{p=0}^{m}{m\choose p}F_{p,n} x^{m-p}=0
\end{align*}
and the result follows.

We prove that $F_{m,n}=0$ for $m<n$ by induction on $n$. This is vacuously true when $n=0$. Assume the claim is true for a given $n$ and consider $F_{m,n+1}$ with $m<n+1$. We have the following chain of equalities
\begin{align*}
F_{m,n+1}&=\sum_{i=0}^{k^{n+1}-1}\zeta_{k}^{s_{k}(i)}\mathbf{u}(i)^{m}=\sum_{j=0}^{k-1}\sum_{i=0}^{k^{n}-1}\zeta_{k}^{s_{k}(ki+j)}\mathbf{u}(ki+j)^{m}\\
            &=\sum_{j=0}^{k-1}\zeta_{k}^{j}\sum_{i=0}^{k^{n}-1}\zeta_{k}^{s_{k}(i)}(P(\mathbf{u}(i))+j q)^{m}\\
            &=\sum_{j=0}^{k-1}\zeta_{k}^{j}\sum_{i=0}^{k^{n}-1}\zeta_{k}^{s_{k}(i)}\sum_{l=0}^{m}{m\choose l}j^{l}q^{l}P(\mathbf{u}(i))^{m-l}\\
            &=\sum_{j=0}^{k-1}\zeta_{k}^{j}\sum_{l=0}^{m}{m\choose l}j^{l}q^{l}\sum_{i=0}^{k^{n}-1}\zeta_{k}^{s_{k}(i)}P(\mathbf{u}(i))^{m-l}\\
            &=\sum_{j=0}^{k-1}\zeta_{k}^{j}\sum_{l=0}^{m}{m\choose l}j^{l}q^{l}P(F_{m-l,n}).
\end{align*}
Observe now that by the induction hypothesis the terms with $m-l<n$ vanish. Hence the sum  consists only of the terms with $l=0$. Thus we get
\begin{equation*}
F_{m,n+1}=\sum_{j=0}^{k-1}\zeta_{k}^{j}P(F_{m,n})=0
\end{equation*}
since $\sum_{j=0}^{k-1}\zeta_{k}^{j}=0$.
\end{proof}
Because of the identity $f_{m,n}=0$ for $m<n$ and Lemma \ref{lem1}, we are interested solely in the computation of $F_{m,n}$ for $m\geq n$

To simplify the formulation of the next result, let us introduce the sequence $$a_n=\sum_{j=0}^{k-1} j^n \zeta_k^j.$$
 Note the following easy special cases. For $k=2$, we have $a_0=0$ and $a_n=-1$ for $n\geq 1$.
 For a general $k$, we have  $a_0=0$, $a_1=\frac{k}{\zeta_k-1}$, and $$a_2=k((k-2)\zeta_k-k)/(\zeta_k-1)^2.$$  

We now state a recurrence relation for the sequence $F_{m,n}$.

\begin{thm}\label{thm1}
The following recurrence holds
\begin{equation}\label{recG}
F_{m,n}=\sum_{r=1}^{m-n+1}a_r{m\choose r}q^{r}P(F_{m-r,n-1}), \quad n\geq 1, m\geq 0.
\end{equation}
\end{thm}
\begin{proof}
Since $u(ki+j)=P(u(i))+jq$, we have
\begin{align*}
F_{m,n+1}&=\sum_{i=0}^{k^{n+1}-1}\zeta_{k}^{s_{k}(i)}{\mathbf u}(i)^{m}\\&\;=\sum_{j=0}^{k-1}\sum_{i=0}^{k^{n}-1}\zeta_{k}^{s_{k}(ki+j)}(P({\mathbf u}(i))+jq)^{m}\\
              &\;=\sum_{j=0}^{k-1}\zeta_{k}^{j}\sum_{i=0}^{k^{n}-1}\zeta_{k}^{s_{k}(i)}\sum_{r=0}^{m}{m\choose r}j^{r}q^{r}P({\mathbf u}(i))^{m-r}\\
              &\;=\sum_{r=0}^{m}a_{r}{m\choose r}q^rP(F_{m-r,n}).
              \end{align*}
The term with $r=0$ vanishes since $a_0=0$; the terms with $r\geq m-n+1$ vanish since by Proposition \ref{lem2} we have $F_{m-r,n}=0$. Thus
\begin{align*}
F_{m,n+1}&=\sum_{r=1}^{m-n} a_{r}{m\choose r}q^{r}P(F_{m-r,n}).
\end{align*}
To conclude, replace $n$ by $n-1$.
\end{proof}

\begin{prop}\label{divisibility}
In the ring $\mathrm{Sym}(V)$, the elements $F_{m,n}$ are divisible by $\prod_{j=0}^{n-1} P^{j}(q)$.\end{prop}
\begin{proof} The proof is by induction on $n$. There is nothing to prove for $n=0$. The claim for a general $n$ follows from the recurrence in Theorem \ref{thm1}. In fact, $\prod_{j=0}^{n-1} P^{j}(q)$ divides all the terms in this recurrence.\end{proof}

In order to study the sequence $F_{m,n}$ further, we introduce a related, simpler sequence  $H_{m,n}$ of elements of $\mathrm{Sym}(V)$ given by

\begin{equation}\label{defH}
H_{m,n}=\frac{(\zeta_{k}-1)^{n}}{k^{n}(n+m)!\prod_{i=0}^{n-1}P^{i}(q)}F_{m+n,n}.
\end{equation}
This definition makes sense, since by Proposition \ref{divisibility} we know that $F_{m,n}$ is divisible by $\prod_{i=0}^{n-1} P^{i}(q)$ and the ring $\mathrm{Sym}(V)$ is a domain. Note that in fact $H_{m,n}$ lie in $ \mathrm{Sym}^{m}(V)$, the $m$th symmetric power of $V$.

We start by rewriting the recurrence relation for $F_{m,n}$ in terms of $H_{m,n}$.

\begin{lem}\label{lem3}
The  sequence $H_{m,n}$ satisfies the following recurrence relation
\begin{equation}\label{recH}
H_{m,n}=P(H_{m,n-1})+\sum_{r=2}^{m+1}\frac{a_{r}(\zeta_{k}-1)}{k r!}q^{r-1} P(H_{m+1-r,n-1}), \quad n\geq 1.
\end{equation}
\end{lem}
\begin{proof}
We obtain this recurrence immediately from Theorem \ref{thm1} by substituting the formula for $H_{m,n}$ and separating the term with $r=1$. 
\end{proof}

We will use the following easy lemma providing a solution to a "twisted" linear recurrence relation of order 1.
\begin{lem}\label{sollinrec}
Suppose that a sequence $r_{n}$ of elements of a ring $R$ satisfies a recurrence relation
$$r_{n}=P(r_{n-1})+b_{n}, \quad n\geq1$$
for some $b_{n} \in R$ and a ring endomorphism $P\colon R\to R$. Then
\begin{equation*}
r_{n}=P^n(r_{0})+\sum_{i=1}^{n}P^{n-i}(b_{i}),
\end{equation*}
where $P^{i}=P\circ P \circ \ldots \circ P$ denotes the $i$-th iteration of $P$.
\end{lem}
\begin{proof} The formula follows immediately by induction on $n$.
\end{proof}

\begin{thm}\label{czapka} We have $$H_{m,n} =\sum_{\substack{m=\nu_1+\ldots+\nu_t\\ \nu_1,\ldots,\nu_t\geq 1}} \left(\prod_{i=1}^t \frac{a_{\nu_i+1}(\zeta_k-1)}{k(\nu_i+1)!}\right) \sum_{0\leq l_1<\ldots<l_t\leq n-1}P^{l_1}(q)^{\nu_1}\cdots P^{l_t}(q)^{\nu_t}. $$
In the formula, the sum is taken over all $t\geq 0$ and all tuples $(\nu_1,\ldots,\nu_t)\in\N^t$ such that $\nu_i\geq 1$ and $\nu_1+\ldots+\nu_t=m$. When $m=0$, one should interpret this formula as saying that $H_{0,n}=1$.
\end{thm}
\begin{proof} Define $H'_{m,n}$ by the formula stated in the theorem. We will prove that $H_{m,n}=H'_{m,n}$. Since $\mathbf{u}(0)=0$, we get $H_{m,0}=H'_{m,0}=0$ for $m\geq 1$ and $H_{0,0}=H'_{0,0}=1$. Applying Lemma \ref{sollinrec} to Lemma \ref{lem3}, we get \begin{align*} H_{m,n}&=\sum_{i=1}^n\sum_{r=2}^{m+1}\frac{a_r (\zeta_k-1)}{kr!}P^{n-i}(q)^{r-1}P^{n-i+1}(H_{m+1-r,i-1})\\&=\sum_{r=2}^{m+1}\frac{a_r (\zeta_k-1)}{kr!}\sum_{i=0}^{n-1}P^{i}(q)^{r-1}P^{i+1}(H_{m+1-r,n-i-1})\end{align*} To end the proof, it is sufficient to note that the sequence $H'_{m,n}$ satisfies the same recurrence. This is easily verified to be the case. In fact, the terms in the recurrence with a given $r$ and $i$ correspond to the terms in the formula with $\nu_1=r-1$  and $l_1=i$.\end{proof}

\begin{cor}\label{closedF} Let $k=2$ and $\mathbf{u}(n)=n$. We then have $$F_{m+n,n} =(-1)^n (n+m)! 2^{\frac{n(n-1)}{2}} \!\!\!\!\!\!\!\!\!\sum_{\substack{m=\nu_1+\ldots+\nu_t\\ \nu_1,\ldots,\nu_t\geq 1}} \left(\prod_{i=1}^t \frac{1}{(\nu_i+1)!}\right) \sum_{0\leq l_1<\ldots<l_t\leq n-1}2^{l_1\nu_1+\ldots+l_t\nu_t}. $$
\end{cor}
\begin{proof} Recall that for $k=2$ and $\mathbf{u}(n)=n$, we have $\zeta_k=-1$, $a_n=-1$ for $n\geq 1$, $q=1$ and $P$ is the multiplication by $2$ map. The formula follows immediately from Theorem \ref{czapka} and the definition of $H_{m,n}$.\end{proof}

For $m=0$ and $m=1$, this reproves the identities of Nieto and Bateman-Bradley. For $m=2$, we get
\begin{equation*}
f_{n+2,n}(x)=(-1)^{n}2^{\frac{n(n-1)}{2}}(n+2)!\left(\frac{1}{2}x^2+\frac{1}{2}(2^{n-1}-1)x+\frac{1}{36}(5\cdot 2^{2n}-9\cdot 2^{n}+4)\right).
\end{equation*}
We state below some special values of $H_{m,n}$ for $m=0,1$ and $F_{m,n}$ for $m=n$ and $m=n+1$.  All these computations follow immediately from Theorem \ref{czapka}.
\begin{cor} We have $H_{0,n}=1$ and $H_{1,n}=c_k\sum_{i=0}^{n-1}P^i(q),$ where $c_k=\frac{(k-2)\zeta_k-k}{2(\zeta_k-1)}$ \textup{(}so that $c_2=1/2$\textup{)}.
\end{cor}

\begin{cor}\label{cor1}
We have
\begin{align*}
F_{n,n}&=d_k n!\prod_{i=0}^{n-1}P^{i}(q),\\
F_{n+1,n}&=e_k (n+1)!\left(\prod_{i=0}^{n-1}P^{i}(q)\right)\sum_{j=0}^{n-1}P^j(q).
\end{align*}
where $$d_k=\frac{k^{n}}{(\zeta_{k}-1)^{n}}, \quad e_k=\frac{k^n((k-2)\zeta_{k}-k)}{2(\zeta_{k}-1)^{n+1}}.$$ In particular, for $k=2$, we have $d_k=(-1)^n, e_k=(-1)^n/2$.
\end{cor}

We apply these result to get some interesting identities involving the sum of digits function $s_2(n)$.

\begin{cor}\label{cor2}
Let $n\geq 0$ be an integer. We then have
\begin{equation*}
\sum_{i=0}^{2^{n}-1}(-1)^{s_{2}(i)}s_{2}(i)^{r}i^{m-r}=0\end{equation*} for $0\leq r \leq m < n$ and \begin{align*}
\sum_{i=0}^{2^{n}-1}(-1)^{s_{2}(i)}s_{2}(i)^{r}i^{n-r}&=(-1)^n r!(n-r)!\sigma_{n-r}(1,2,\ldots,2^{n-1})
\end{align*}
 for $0\leq r \leq n$, where $\sigma_{i}$ is the $i$-th symmetric polynomial in $n$ variables.
\end{cor}

\begin{proof} Consider the sequence ${\mathbf u}(n)=s_{2}(n)v_1+nv_2$ with values in a two dimensional vector space $V$ with basis $v_1, v_2$. This sequence satisfies the recurrence (\ref{czajnik}) with $q=v_1+v_2$ and $P$  given by the matrix $\bigl(\begin{smallmatrix} 1&0\\ 0&2 \end{smallmatrix} \bigr)$ (in the basis $v_1,v_2$). The sum in the theorem is equal to the coefficient of $v_{1}^rv_{2}^{n-r}$ in $F_{m,n}^{\mathbf{u}}$ in the first case and the coefficient of $v_1^rv_2^{n-r}$ in $F_{n,n}^{\mathbf{u}}$ in the second case. We get the claim by Proposition \ref{lem2} and Corollary \ref{cor1}, respectively.
\end{proof}

We can obtain similar results for different choices of the sequence $\mathbf{u}$. We give one more example below.

\begin{cor}\label{cor3} We have
\begin{equation*}
\sum_{i=0}^{2^{n}-1}(-1)^{s_{2}(i)}(2^{r}s_{2}(i)-i)^{m}=0,
\end{equation*}
for integers $n,m\geq0$ and $0 \leq r \leq n-1$.
\end{cor}
\begin{proof} As in Corollary $\ref{cor2}$, we consider the sequence ${\mathbf u}(n)=s_{2}(n)v_1+nv_2$ with values in a two dimensional vector space $V$ with basis $v_1, v_2$. We have $P^r(q)=v_1+2^r v_2$ for $0 \leq r \leq n-1$. Consider the linear map $\varphi \colon V \to K$ such that $\varphi(v_1)=2^r$, $\varphi(v_2)=-1$. This map extends to a homomorphism of $K$-algebras $\varphi\colon \mathrm{Sym}(V)\to K$ and $\varphi(P^r(q))=0$. The expression in the corollary is equal to $\varphi(F_{m,n})$. Since by Proposition \ref{divisibility} we know that $F_{m,n}$ is divisible by $P^r(q)$, we get $$\varphi(F_{m,n})=\sum_{i=0}^{2^{n}-1}(-1)^{s_{2}(i)}(2^{r}s_{2}(i)-i)^{m}=0.$$ \end{proof}

We have already seen in Lemma \ref{lem3} that the sequence $H_{m,n}$ satisfies a certain recurrence relation. We will now prove that it also satisfies a linear recurrence in $n$.

\begin{thm}\label{uscor} Consider the induced action of $P$ on the space $$S_m=\oplus_{d=0}^m \mathrm{Sym}^d V.$$ This space has dimension $M=\dim S_m = \binom{m+\dim V}{m}$. Let $$\chi_{P,m}(y)=\det(\mathrm{id}-Py\mid S_m)$$ be the reciprocal of the characteristic polynomial of the action of $P$ on $S_m$. \textup{(}In particular, $\deg\chi_{P,m} \leq M$ and equality holds when $P$ is invertible.\textup{)} Put $$H_m = \sum_{n\geq 0} H_{m,n}y^n$$ \textup{(}this is a power series with coefficients in $\mathrm{Sym}^m(V)$\textup{)}. Then $H_m$ is a rational function or more precisely $$H_m=\frac{h(y)}{\chi_{P,m}(y)}$$ for some polynomial $h(y)$ with coefficients in $\mathrm{Sym}^m(V)$ of degree $\deg h < M$.
\end{thm}

\begin{cor} In the notation of the preceding theorem, write
  $\chi_{P,m}(y)=\sum_{j=0}^{M} b_j y^j$ with $b_j \in K$. Then $$H_{m,n}=-\sum_{j=1}^{M} b_j H_{m,n-j},\quad n\geq M.$$
\end{cor}

The corollary follows immediately from Theorem \ref{uscor}. Before we begin the proof of the theorem, note that the statement simplifies greatly when $V$ is the field $K$ and $P\colon V\to V$ is multiplication by $p$. In this case, $H_{m,n}$ satisfies a linear relation in $n$ of degree $m+1$ whose coefficients are equal to the coefficients of the polynomial $\prod_{i=0}^{m}(1-p^i y)$. It might be useful to think about this case when reading the following argument.

\begin{proof} We do induction on $m$. For $m=0$, we have $H_{0,n}=1$ and $H_0=1/(1-y)$ and so the claim holds in this case. Assume now that $m\geq 1$.

By Lemma $\ref{lem3}$, we have $$H_{m,n}=P(H_{m,n-1})+\sum_{r=2}^{m+1}c_r q^{r-1} P(H_{m+1-r,n-1})$$ for some $c_r\in K$. Multiplying this equality by $y^n$ and summing over all $n\geq 1$, we get $$H_m=yP(H_m)+\sum_{r=2}^{m+1} c_r q^{r-1}y P(H_{m+1-r}).$$ (Recall that $H_{m,0}=0$ for $m \geq 1$.) We obtain, by induction on $m$, the following equality
\begin{equation}\label{pietuszki}
H_m-yP(H_m)=\sum_{r=2}^{m+1} c_r q^{r-1}y P(H_{m+1-r})=\frac{g(y)}{\chi_{P,m-1}(y)}
\end{equation}
 for some $g(y)\in\mathrm{Sym}^m(V) [y]$ of degree $\deg g\leq M_{m-1} = \dim S_{m-1}$. (We use here the obvious fact that $\chi_{P,m_1}$ divides $\chi_{P,m_2}$ for $m_1\leq m_2$.) If $\dim V=1$ and $P$ is the multiplication by $p$ map, the left hand side is just $(1-p^my)H_m$ and we get the result simply after dividing by $(1-p^m y)$. In the case $\dim V>1$, we have to proceed more carefully.  Consider $N=K[\![y]\!]\otimes_K \mathrm{Sym}^m(V)$ with a structure of a $K[\![y]\!][t]-$module, where multiplication by $t$ is induced by the action of $P$ on $\mathrm{Sym}^m(V)$. As a $K[\![y]\!]-$module, $N$ is free of rank $\dim \mathrm{Sym}^m(V)$. Let $N'$ be the $K[\![y]\!][t]-$submodule of $N$ generated by $H_m$. Then, as a $K[\![y]\!]$-module, $N'$ is generated by $v_i=P^i(H_m)$ for $0\leq i \leq l-1$, where $l=\dim \mathrm{Sym}^m(V)$. We can write  $$tv_i=\sum_{j=0}^{l-1}a_{ij}v_j$$ for a certain $l\times l$ matrix $A=(a_{ij})$ with coefficients in $K$. By equation (\ref{pietuszki}), we have $$(I-Ay)\begin{pmatrix}v_0\\\vdots\\v_{l-1}\end{pmatrix}=\begin{pmatrix}w_0\\\vdots\\w_{l-1}\end{pmatrix},$$ where $w_i$ are power series of the form $\frac{g_j(y)}{\chi_{P,m-1}(y)}$ for some polynomials $g_j$ of degree $\deg g_j\leq M_{m-1}$. Consider now the adjoint matrix $(I-Ay)^{\adj}$ of the matrix $I-Ay$. Its coefficients are polynomials in $K[y]$ of degree smaller than $\dim \mathrm{Sym}^m(V)$. We then have $(I-Ay)^{\adj}(I-Ay) = \det(\id-Py\mid \mathrm{Sym}^m(V)) I$ and so $$\det(\id-Py\mid \mathrm{Sym}^m(V)) \begin{pmatrix}v_0\\\vdots\\v_{l-1}\end{pmatrix} = (I-Ay)^{\adj}\begin{pmatrix}w_0\\\vdots\\w_{l-1}\end{pmatrix}.$$ We can write $$ \begin{pmatrix}v_0\\\vdots\\v_{l-1}\end{pmatrix} = \det(\id-Py\mid \mathrm{Sym}^m(V))^{-1}(I-Ay)^{\adj}\begin{pmatrix}w_0\\\vdots\\w_{l-1}\end{pmatrix}.$$ Now, $w_i$ are power series with coefficients  $\frac{g_j(y)}{\chi_{P,m-1}(y)}$ for $g_j$ of degree at most $M_{m-1}$, $(I-Ay)^{\adj}$ has as its coefficients polynomials in $y$ of degree smaller than $l=\dim \mathrm{Sym}^m(V)$ and $$\chi_{P,m-1}(y)\det(\id-Py\mid \mathrm{Sym}^m(V))=\chi_{P,m}(y).$$ Hence all $v_i$ are rational functions of the form $\frac{h_i(y)}{\chi_{P,m}(y)}$ with polynomials $h_i(y)$ of degree smaller than $M_{m-1}+l=M$. This applies in particular to $v_0=H_m$, which is exactly what we wanted to prove.
\end{proof}

\section{Identities involving Prouhet-Thue-Morse like sequences}\label{Section3}

The property of the Prouhet-Thue-Morse sequence $t_{n}$ which says that $\sum_{i=0}^{2^n-1}t_{i}i^{m}=0$ for $m=0,1,\ldots,n-1$ implies the existence of sets $P_n, Q_{n}$ satisfying the following properties:
\begin{equation*}
 P_{n}\cup Q_{n}=\{0,1,\ldots,2^{n}-1\},\quad P_{n}\cap Q_{n}=\emptyset
\end{equation*}
and
\begin{equation}\label{partition}
 \sum_{i\in P_{n}}i^{m}=\sum_{i\in Q_{n}}i^{m}
\end{equation}
for each $m\in\{0,1,\ldots,n-1\}$. In fact, it is enough to take $$P_n=\{i\in\{0,1,\ldots,2^n-1\}\colon t_i=1\}$$ and $$Q_n=\{i\in\{0,1,\ldots,2^n-1\}\colon t_i=-1\}.$$ In particular, $P_{n}$ and $Q_{n}$ have the same number of elements (take m=0). A question arises whether for a certain $r$ it is possible to construct a partition of $\{0,1,\ldots,r\}$ into two sets of equal cardinalities so that a property analogous to (\ref{partition}) holds. For an interesting approach to this problem one can consult \cite{Cer}. We are especially interested in the value of $r$ of the form $k^n-1$; in particular, $k$ needs to be even since $P_n$ and $Q_n$ are supposed to be equinumerous. To state our results, we introduce the following notation.
Let $k$ be an integer $\geq 2$ and put $I_{k}=\{0,1,\ldots,k-1\}$. For any $l\in I_{k}$, let
\begin{center}
$N_{k}(l,i)$=number of occurrences of a digit $l$ in the expansion of \\an integer $i$ in base $k$.
\end{center}
 Moreover, we observe that for $l\in\{0,1,\ldots,k-1\}$ the function $N_{k}(l,i)$ satisfies the following recurrence relations
\begin{equation*}
 N_{k}(l,ki+j)=\begin{cases}
\begin{array}{lll}
N_{k}(l,i)   &  & \mbox{for}\;j\in I_{k}\setminus\{l\}, \\
N_{k}(l,i)+1 &  & \mbox{for}\;j=l
                            \end{array}
\end{cases}
\end{equation*}

Using this recurrence relation, one can easily prove the following useful result.

\begin{lem}\label{iden}
 Let $n$ and $k\geq 2$ be integers and $t_{0},\ldots,t_{k-1}$ be variables. Then the following identity holds
\begin{equation}\label{prod:iden}
 \prod_{i=0}^{n-1}\Big(\sum_{j=0}^{k-1}t_{j}x^{jk^{i}}\Big)=\sum_{i=0}^{k^{n}-1}\Big(\prod_{l=0}^{k-1}t_{i}^{N_{k}(l,i)}\Big)x^{i}.
\end{equation}
\end{lem}
\begin{proof} Regroup the terms.\end{proof}

We are ready to state the following result.
Let $b$ be a positive integer and $k=2b$. Put $A_{b}=\{1,2,\ldots,2b-1\}$ and consider the set
$$\cal{A}_{b}=\{(c_{1},c_{2},\ldots c_{b})\in A_{b}^{b}:\;c_{i}\neq c_{j}\;\mbox{for}\;i\neq j\}.$$
In particular, $|\cal{A}_{b}|={2b-1\choose b}$. For a given $v=(v_{1},\ldots,v_{b})\in\cal{A}_{b}$, we put
\begin{equation*}
N_{k,v}(i)=\sum_{j=1}^{b}N_{k}(v_{j},i).
\end{equation*}

\begin{thm}\label{PTMlike} For integers $n>m\geq 0$, we have
$$
\sum_{i=0}^{k^{n}-1}(-1)^{N_{k,v}(i)}i^{m}=0.
$$
In particular, for an integer $n$ and an even $k\geq 2$, the set $\{0,1,\ldots,k^n-1\}$ can be partitioned into two disjoint subsets $P$ and $Q$  such that
$$
\sum_{i\in P}i^{m}=\sum_{i\in Q}i^{m}
$$
for each $m=0,1,\ldots,n-1$.
\end{thm}
\begin{proof}
In order to get the result, we will use  Lemma \ref{iden}. More precisely, let $v=(v_{1},\ldots,v_{b})\in\cal{A}_{b}$  be given. Put $V=\{v_{1},\ldots,v_{b}\}$. In the expression in Lemma \ref{iden}, we substitute
$$
t_{i}=\begin{cases}
\begin{array}{lll}
-1 &  &\mbox{for}\;i\in V  \\
+1 &  &\mbox{for}\;i\in\cal({A}_{b} \cup \{0\}) \setminus V.
\end{array}
\end{cases}
$$
For such $t_{i}$, we have $\sum_{j=0}^{k-1}t_{j}=0$ which implies that for each $0\leq i \leq n-1$ the polynomial $h_{i}(x)=\sum_{j=0}^{k-1}t_{j}x^{jk^{i}}$ has a root at $x=1$. This shows that the polynomial $F_{n}(x)=\prod_{i=0}^{n-1}h_{i}(x)$ has a root at $x=1$ with multiplicity at least $n$, so that $F_{n}(x)=(x-1)^{n}W_n(x)$. On the other hand,
$$
F_{n}(x)=\sum_{i=0}^{k^{n}-1}(-1)^{N_{k,v}(i)}x^{i}.
$$
Let $\theta$ be the differential operator defined by $\theta =x\frac{d}{dx}$, i.e., $\theta \Big(\sum_{i=0}^{s}a_{i}x^{i}\Big)=\sum_{i=0}^{s}ia_{i}x^{i}$. Applying $m$ times the operator $\theta$ to the polynomial $F_{n}$, we see that
$$
\theta^{m}(F_{n}(x))=\sum_{i=0}^{k^{n}-1}(-1)^{N_{k,v}(i)}i^{m}x^{i}
$$
and since $F_n(x)$ has a root with multiplicity at least $n$ at $x=1$, $\theta^{m}(F_{n}(x))$ has a root of order at least $n-m$ at $x=1$.  This shows that $\sum_{i=0}^{k^{n}-1}(-1)^{N_{k,v}(i)}i^{m}=0$ for $m=0,\ldots,n-1$. For the final statement, take $$P=\{i \in \{0,\ldots, k^n-1\} \colon (-1)^{N_{k,v}(i)}=1\}$$ and $$Q=\{i \in \{0,\ldots, k^n-1\} \colon (-1)^{N_{k,v}(i)}=-1\}.$$
\end{proof}

Using Theorem \ref{PTMlike} and the well known fact that $\binom x m$ can be written as a linear combination of $1, x, \ldots, x^{m}$, we easily get the following

\begin{cor}\label{nasturcja}
Let $v=(v_{1},\ldots,v_{b})\in\cal{A}_{b}$. Consider the polynomial
\begin{equation*}
f_{m,n}(v,x)=\sum_{i=0}^{k^{n}-1}(-1)^{N_{k,v}(i)}(x+i)^{m}.
\end{equation*}
Then:
\begin{enumerate}
\item If $m<n$, then $f_{m,n}(v,x)=0$.
\item If $m\geq n$, then  $\mathrm{deg}f_{m,n}(v,x)\leq m-n$. In particular,  $f_{n,n}(v,x)$ is a constant polynomial.
\end{enumerate}
 \end{cor}
\begin{proof}
The first part of our corollary is an immediate consequence of Theorem \ref{PTMlike}. In order to prove the second part, we write
\begin{align*}
\sum_{i=0}^{k^{n}-1}(-1)^{N_{k,v}(i)}(x+i)^{m}&=\sum_{i=0}^{k^{n}-1}(-1)^{N_{k,v}(i)}\sum_{j=0}^{m}{m\choose j}x^{m-j}i^{j}\\
                                              &=\sum_{j=0}^{m}{m\choose j}x^{m-j}\sum_{i=0}^{k^{n}-1}(-1)^{N_{k,v}(i)}i^{j}.
\end{align*}
\end{proof}

Put $F_{n,n}(v)=f_{n,n}(v,x)$.
\begin{prop}
We have $$F_{n,n}(v)=n! k^{\binom{n}{2}}\left(\sum_{\substack{0\leq j\leq k-1\\j\notin v}}j-\sum_{\substack{0\leq j\leq k-1\\j\in v}}j\right)^n.$$ \end{prop}
\begin{proof} We will only sketch the argument since it resembles the proof of Proposition \ref{lem2}. We have
\begin{align*}F_{n+1,n+1}(v)&=\sum_{i=0}^{k^{n+1}-1}(-1)^{N_{k,v}(i)}i^{n+1}
            =\sum_{j=0}^{k-1}\sum_{i=0}^{k^{n}-1}(-1)^{N_{k,v}(ki+j)}(ki+j)^{n+1}\\
            &=\sum_{\substack{0\leq j\leq k-1\\j\notin v}}\sum_{i=0}^{k^{n}-1}(-1)^{N_{k,v}(i)}(ki+j)^{n+1}-\sum_{\substack{0\leq j\leq k-1\\j\in v}}\sum_{i=0}^{k^{n}-1}(-1)^{N_{k,v}(i)}(ki+j)^{n+1}
\end{align*}
since $N_{k,v}(ki+j)=N_{k,v}(i)$ if $j\notin v$ and $N_{k,v}(ki+j)=N_{k,v}(i)+1$ if $j\in v$.
Note now that for a fixed $j$ we have
\begin{align*}
\sum_{i=0}^{k^{n}-1}(-1)^{N_{k,v}(i)}(ki+j)^{n+1}&=\sum_{s=0}^{n+1}\sum_{i=0}^{k^{n}-1}(-1)^{N_{k,v}(i)}\binom{n+1}{s}k^si^sj^{n+1-s}\\&=\sum_{s=0}^{n+1} \binom{n+1}{s} j^{n+1-s}k^s f_{s,n}(v,0).
\end{align*} The terms with $0\leq s<n$ vanish by Corollary \ref{nasturcja}. Therefore we have
\begin{align*}
F_{n+1,n+1}(v)&=(n+1)k^n\left(\sum_{\substack{0\leq j\leq k-1\\j\notin v}}j-\sum_{\substack{0\leq j\leq k-1\\j\in v}}j\right)F_{n,n}(v) \\&+ k^{n+1}\left(\sum_{\substack{0\leq j\leq k-1\\j\notin v}}1-\sum_{\substack{0\leq j\leq k-1\\j\in v}}1\right)f_{n+1,n}(v,0).
\end{align*} The latter sum vanishes since exactly half of $0\leq j \leq k-1$ lie in $v$. The claim follows.
\end{proof}

\bigskip




\section{Open questions and conjectures}\label{Section4}

In this section, we propose some related questions and conjectures. Motivated by the result obtained in Proposition \ref{nasturcja}, we state the following problem.

\begin{prob}
Present a method which allows to compute a closed expression for the polynomial
\begin{equation*}
f_{m,n}(v,x)=\sum_{i=0}^{k^{n}-1}(-1)^{N_{k,v}(i)}(x+i)^{m},
\end{equation*}
for integers $m>n$ and $v\in\cal{A}_{b}$, where $\cal{A}_{b}$ is as in Theorem \ref{PTMlike}.
\end{prob}

Next we state two conjectures which can be seen as generalizations of some results from Section \ref{Section2}.

\begin{conj}
Let $k\geq 1$ and $a_{1},\ldots,a_{k}$ be nonnegative integers. Put $A=(a_{1},\ldots,a_{k})$ and $X=(x, x_{1}, \ldots, x_{k})$ be a vector of variables. Consider the polynomial
\begin{equation*}
G_{A}(X)=\sum_{i_{1}=0}^{2^{a_{1}}-1}\sum_{i_{2}=0}^{2^{a_{2}}-1}\ldots \sum_{i_{k}=0}^{2^{a_{k}}-1}(-1)^{\sum_{j=1}^{k}s_{2}(i_{j})}\Big(x+\sum_{j=1}^{k}i_{j}x_{j}\Big)^{\sum_{j=1}^{k}a_{j}}.
\end{equation*}
Then we have the following identity
\begin{equation*}
G_{A}(X)=(-1)^{\sum_{j=1}^{k}a_{j}}2^{\sum_{j=1}^{k}\frac{a_{j}(a_{j}-1)}{2}}\Big(\sum_{j=1}^{k}a_{j}\Big)!\Big(\prod_{j=1}^{k}x_{j}^{a_{j}}\Big).
\end{equation*}
\end{conj}

\begin{conj}
Let $m\geq 1$ and $a_{1},\ldots,a_{k}$ be nonnegative integers. Put $A=(a_{1},\ldots,a_{k})$ and $X=(x_{1}, \ldots, x_{m}), Y=(y_{1}, \ldots, y_{m})$ be vectors of variables. Consider the polynomial
\begin{equation*}
G_{A}(X,Y)=\sum_{i_{1}=0}^{k^{a_{1}}-1}\sum_{i_{2}=0}^{k^{a_{2}}-1}\ldots \sum_{i_{k}=0}^{k^{a_{m}}-1}\zeta_{k}^{\sum_{j=1}^{m}s_{k}(i_{j})}\Big(\sum_{j=1}^{m}(s_{k}(i_{j})x_{j}+i_{j}y_{j})\Big)^{\sum_{j=1}^{m}a_{j}}.
\end{equation*}
Then:
\begin{enumerate}
\item If $k=2$, then \begin{equation*}
G_{A}(X,Y)=(-1)^{\sum_{j=1}^{m}a_{j}}\Big(\sum_{j=1}^{m}a_{j}\Big)!
\prod_{j=1}^{m}\Big(\prod_{i_{j}=0}^{a_{j}-1}(x_{j}+2^{i_{j}}y_{j})\Big).
\end{equation*}

\item If $k>2$, then
\begin{equation*}
\prod_{j=1}^{m}\Big(\prod_{i_{j}=0}^{n_{j}-1}(x_{j}+k^{i_{j}}y_{j})\Big)\Big|G_{A}(X,Y).
\end{equation*}
\end{enumerate}
\end{conj}

The next conjecture asks about a different generalization of some of the  results obtained in Section \ref{Section3}.

\begin{conj}
Let $m, n\geq 1$ be integers, $t$ be a variable and put $X=(x_{1},\ldots,x_{m})$. Consider the polynomial
\begin{equation*}
H_{m,n}=H_{m,n}(X,t)=\sum_{i_{1}=0}^{2^{n}-1}\sum_{i_{2}=0}^{2^{n}-1}\ldots \sum_{i_{m}=0}^{2^{n}-1}(-1)^{s_{2}(\sum_{j=1}^{m}i_{j})}\Big(t+\sum_{j=1}^{m}i_{j}x_{j}\Big)^{n}.
\end{equation*}
We then have $\op{deg}_{t}H_{m,n}(X,t)=m-1$. In particular,
\begin{align*}
&H_{1,n}=(-1)^{n}n!2^{\frac{n(n-1)}{2}}x_{1}^{n},\\
&H_{2,n}=(-1)^{n}n!2^{\frac{n(n-1)}{2}}\Big(2\frac{x_{1}^{n}-x_{2}^{n}}{x_{1}-x_{2}}t+2^{n}\frac{x_{1}^{n+1}-x_{2}^{n+1}}{x_{1}-x_{2}}+x_{1}x_{2}(2^{n}-1)\frac{x_{1}^{n-1}-x_{2}^{n-1}}{x_{1}-x_{2}}\Big).
\end{align*}
\end{conj}


\bigskip

\noindent Jagiellonian University, Faculty of Mathematics and Computer Science, Institute of Mathematics, {\L}ojasiewicza 6, 30 - 348 Krak\'{o}w, Poland;

\noindent email: {\tt \{jakub.byszewski, maciej.ulas\}@uj.edu.pl}

 \end{document}